\documentclass[a4paper,12pt,draft]{amsart}
\usepackage[margin=30mm]{geometry}
\usepackage{amssymb, amsmath, amsthm, amscd}

\usepackage[T1]{fontenc}
\usepackage{eucal,mathrsfs,dsfont}%gives nicer set-names. Use \mathds instead of \mathds
\usepackage{color}%cyan,red;magenta,green;yellow,blue
\definecolor{mno}{rgb}{0.5,0.1,0.5}

\newcommand{\R}{\mathds R}

\newcommand{\I}{\mathds 1}

\newtheorem{theorem}{Theorem}[section]
\newtheorem{lemma}[theorem]{Lemma}
\newtheorem{proposition}[theorem]{Proposition}
\newtheorem{corollary}[theorem]{Corollary}

\theoremstyle{definition}

\newtheorem{remark}[theorem]{Remark}

\begin{document}

\title[Harnack Inequalities for O-U Processes Driven by L\'{e}vy Processes]{\bfseries Harnack Inequalities for Ornstein-Uhlenbeck Processes Driven by L\'{e}vy Processes}

\thanks{\emph{J.\ Wang:}
School of Mathematics and Computer Science, Fujian Normal
University, 350007, Fuzhou, P.R. China \emph{and} TU Dresden,
Institut f\"{u}r Mathematische Stochastik, 01062 Dresden, Germany.
\texttt{jianwang@fjnu.edu.cn}}

\date{}

\maketitle

\begin{abstract} By using the existing sharp estimates of density
function for rotationally invariant symmetric $\alpha$-stable
L\'{e}vy processes and rotationally invariant symmetric truncated
$\alpha$-stable L\'{e}vy processes, we obtain that Harnack
inequalities hold for rotationally invariant symmetric
$\alpha$-stable L\'{e}vy processes with $\alpha\in(0,2)$ and
Ornstein-Uhlenbeck processes driven by rotationally invariant
symmetric $\alpha$-stable L\'{e}vy process, while logarithmic
Harnack inequalities are satisfied for rotationally invariant
symmetric truncated $\alpha$-stable L\'{e}vy processes.

\medskip


\medskip

\end{abstract}

\section{Main Results}\label{section1}
The dimension free Harnack inequality for diffusion semigroups was
first established in (Wang, 1997) under a curvature condition. Using
the coupling method and Girsanov transformations, this inequality
has been derived in (Arnaudon et al., 2006) for diffusions with
curvature unbounded below, also see e.g.\ (Wang, 2007, 2010b;
Ouyang, 2009) and references therein for recent work on this topic.
These methods are applied in (R\"{o}ckner and Wang, 2003; Ouyang et
al., 2011) to study Harnack inequalities for a class of
Ornstein-Uhlenbeck type processes driven by L\'{e}vy process with
non-degenerate Gaussian noise on Hilbert spaces. In this paper we
will establish explicit Harnack inequalities for Ornstein-Uhlenbeck
processes driven by pure L\'{e}vy jump processes. For this aim, we
will make full use of the existing estimates of density functions
for L\'{e}vy processes.

\medskip

Let $L_t$ be a $d$-dimensional L\'{e}vy process starting from the
origin and $A$ be a $d\times d$ matrix. The corresponding
\emph{Ornstein-Uhlenbeck process} is defined as the unique strong
solution of the stochastic differential equation
$$dX_t=AX_tdt+dL_t.$$ Denote by $P_t$ the semigroup of $X_t$, and by $B_b^+(\R^d)$ the set of all bounded and nonnegative
measurable functions on $\R^d$. Our main contribution is as follows:

\begin{theorem}\label{section1th1} Assume that the L\'{e}vy measure of $L_t$ satisfies
$$\nu(dz)\ge c|z|^{-d-\alpha}dz,$$ where
$\alpha\in (0,2)$ and $c>0$. Then there exists $C>0$ (only depending
on $d$, $\alpha$, $c$ and $\|A\|$) such that for any $t>0$, $x$, $y\in\R^d$ and $f\in
B_b^+(\R^d)$,
\begin{equation}\label{section1th1222} P_tf(x)\le
CP_tf(y)\bigg(1+\frac{|x-y|}{(t\wedge1)^{1/\alpha}}\bigg)^{d+\alpha}.\end{equation}
If moreover $A=0$, then
\begin{equation*}\label{section1th122222}P_tf(x)\le
CP_tf(y)\bigg(1+\frac{|x-y|}{t^{1/\alpha}}\bigg)^{d+\alpha}.\end{equation*}\end{theorem}

\begin{remark} We shall point out that the above type of Harnack inequality does not hold for Brownian motion.
That is, let $P_t$ be the semigroup of Brownian motion on $\R^d$.
Then, there does not exist a finite measurable function $\Phi$ on
$(0,\infty)\times \R^d\times \R^d$ such that for any $t>0$, $x$,
$y\in\R^d$ and $f\in B_b^+(\R^d)$,
\begin{equation*}P_tf(x)\le P_tf(y) \Phi(t,x,y).\end{equation*}
Indeed, such inequality is not true for most diffusion semigroups on
noncompact manifolds, e.g.\ see (Wang, 2006a).
\end{remark}

\medskip

By the H\"{o}lder inequality, Theorem \ref{section1th1} immediately
yields the following result, which indicates that \emph{F.-Y. Wang's
Harnack inequalities} hold for rotationally invariant symmetric
$\alpha$-stable L\'{e}vy processes with $\alpha\in(0,2)$ and
Ornstein-Uhlenbeck processes driven by rotationally invariant
symmetric $\alpha$-stable L\'{e}vy process.

\begin{corollary}\label{section1co1} Under the condition in Theorem \ref{section1th1},
there exists a constant $C>0$ such that for any $t>0$, $p>1$, $x$,
$y\in\R^d$ and $f\in B_b^+(\R^d)$,
\begin{equation}\label{section1th12223333}\Big(P_tf(x)\Big)^p\le
CP_tf^p(y)\bigg(1+\frac{|x-y|}{(t\wedge1)^{1/\alpha}}\bigg)^{p(d+\alpha)}.\end{equation}
If moreover $A=0$, then
$$\Big(P_tf(x)\Big)^p\le
CP_tf^p(y)\bigg(1+\frac{|x-y|}{t^{1/\alpha}}\bigg)^{p(d+\alpha)}.$$\end{corollary}

Comparing with (Gordina et al., 2011, Theorem 1.1) where a dimension
free Harnack inequality is presented for $\alpha\in (1,2)$, the
advantage of our results is that they hold for any $p\ge1$ and
$\alpha\in (0,2)$. On the other hand, however, our results are
dimension-dependent so that they can not be extended to infinite
dimensions. Next, (Gordina et al., 2011, Theorem 1.4) includes a
dimension-free log-Harnack inequality for $\alpha\in (0,1]$. Below
we present a simpler but dimension-dependent version of this
inequality for a truncated $\alpha$-stable L\'{e}vy process. We
refer to (Bobkov et al., 2001; Wang, 2010a) for the background and
some properties of logarithmic Harnack inequalities.

\begin{theorem}\label{section1th2} Assume that the L\'{e}vy measure of $L_t$ satisfies
\begin{equation*}\label{1proff2}\nu(dz)\ge c|z|^{-d-\alpha}\I_{\{|z|\le r\}}dz,\end{equation*} where
$\alpha\in (0,2)$ and $c$, $r\in (0,\infty)$. Let $P_t$ be the
semigroup of $L_t$. Then there exists $C>0$ (only depending on $d$,
$\alpha$, $c$ and $r$) such that for any $t>0$, $x$, $y\in\R^d$ and $f\in B_b^+(\R^d)$ with
$f\ge1$,
$$P_t(\log f)(x)\le \log P_tf(y)+C\big(1+|x-y|\big)\log\bigg(\frac{2+|x-y|}{t\wedge1}\bigg).$$\end{theorem}

\medskip

\noindent {\bf Note}\,\, The classical Harnack's inequality is an
inequality relating the values of a positive harmonic function at
two points. In this paper, we also call \eqref{section1th1222} or
\eqref{section1th12223333} the Harnack inequality. The notation,
different from the classical one, is due to Prof. F.-Y. Wang's work.
The readers can refer to the survey paper (Wang, 2006b) for more
details of Wang's Harnack inequality.

\section{Proofs}

The idea of the proofs of Theorems \ref{section1th1} and
\ref{section1th2} is to compare the original process with rotationally invariant symmetric $\alpha$-stable L\'{e}vy processes and with rotationally invariant symmetric truncated $\alpha$-stable
L\'{e}vy processes, respectively.

\subsection{Proof of Theorem \ref{section1th1}}  We first present some preliminary results concerning
rotationally invariant symmetric $\alpha$-stable L\'{e}vy process $X_t$, which is a pure jump
L\'{e}vy process on $\R^d$ with L\'{e}vy measure
\begin{equation}\label{st1}\nu(dz)=c|z|^{-(d+\alpha)}dz,\end{equation}
where $c>0$ and $\alpha\in(0,2).$

According to (Blumenthal and Getoor, 1960, Theorem 2.1), we know
that the process $X_t$ has a continuous density function $p_t(x,y)$,
and
 there exist constants $c_1$ and $c_2>0$ that depend
only on $d$, $\alpha$ and the constant $c$ in \eqref{st1} such that
for all $x$, $y\in\R^d$ and $t>0$,
\begin{equation}\label{st2}c_1\min\bigg\{t^{-d/\alpha}, \frac{t}{|x-y|^{d+\alpha}}\bigg\}\le p_t(x,y)\le c_2\min\bigg\{t^{-d/\alpha}, \frac{t}{|x-y|^{d+\alpha}}\bigg\}. \end{equation}

\medskip

The following result is based on \eqref{st2}, and it is a key in the proof of Theorem \ref{section1th1}.
\begin{lemma}\label{levt} Let $p_t(x,y)$ be the density function of the rotationally invariant symmetric $\alpha$-stable L\'{e}vy process $X_t$ above.
Then, for any $t>0$ and $x$, $y$, $z\in\R^d$,
\begin{equation}\label{st3}\frac{p_t(x,z)}{p_t(y,z)}\le
\frac{2^{\alpha+d}c_2}{c_1}\bigg(1+\frac{|x-y|}{t^{1/\alpha}}\bigg)^{d+\alpha},
\end{equation}
\end{lemma}

\begin{proof}
(i) When $|y-z|\le t^{1/\alpha},$
$$t^{-d/\alpha}\le\frac{t}{|y-z|^{\alpha+d}},$$ and so $$p_t(y,z)\ge
c_1t^{-d/\alpha}.$$ Then,
$$\frac{p_t(x,z)}{p_t(y,z)}\le \frac{c_2t^{-d/\alpha}}{ c_1t^{-d/\alpha}}=\frac{c_2}{c_1}.$$

(ii) If $|y-z|\ge 2(t^{1/\alpha}\vee |x-y|)$, then $$|x-z|\ge
|y-z|-|x-y|\ge |y-z|-|y-z|/2=|y-z|/2\ge t^{1/\alpha}.$$ We get
$$t^{-d/\alpha}\wedge \frac{t}{|x-z|^{\alpha+d}} = \frac{t}{|x-z|^{\alpha+d}}\le  \frac{2^{\alpha+d}t}{|y-z|^{\alpha+d}}.$$ On the other hand,
$$t^{-d/\alpha}\wedge \frac{t}{|y-z|^{\alpha+d}} =
\frac{t}{|y-z|^{\alpha+d}}.$$ Therefore,
$$\frac{p_t(x,z)}{p_t(y,z)}\le \frac{c_22^{\alpha+d}}{ c_1}.$$

(iii) If $|y-z|\le 2(t^{1/\alpha}\vee |x-y|)$ and
$|y-z|>t^{1/\alpha},$ then
$$t^{-d/\alpha}\wedge \frac{t}{|y-z|^{\alpha+d}} =
\frac{t}{|y-z|^{\alpha+d}}.$$ So,
$$\frac{p_t(x,z)}{p_t(y,z)}\le\frac{c_2t^{-d/\alpha}}{c_1t/|y-z|^{\alpha+d}}=
\frac{c_2|y-z|^{\alpha+d}}{c_1t^{(d+\alpha)/\alpha}}.$$ Noticing
that in this case $$\frac{|y-z|}{t^{1/\alpha}}\le
2\Big(1+\frac{|x-y|}{t^{1/\alpha}}\Big),$$ we have
$$\frac{p_t(x,z)}{p_t(y,z)}\le
\frac{2^{\alpha+d}c_2}{c_1}\bigg(1+\frac{|x-y|}{t^{1/\alpha}}\bigg)^{d+\alpha}.$$

Combining with all the estimates above, we can arrive at
\eqref{st3}. The proof is completed.
  \end{proof}

Having Lemma \ref{levt} at hand, we can easily obtain the following
statement for rotationally invariant symmetric $\alpha$-stable L\'{e}vy processes.

\begin{proposition}\label{levtsss} Let $P_t$ be the semigroup of the rotationally invariant symmetric $\alpha$-stable L\'{e}vy process $X_t$ above.
Then there exists $C>0$ such that for any $t>0$, $x$, $y\in\R^d$ and
$f\in B_b^+(\R^d)$,
\begin{equation}\label{strongharnack}P_tf(x)\le
CP_tf(y)\bigg(1+\frac{|x-y|}{t^{1/\alpha}}\bigg)^{d+\alpha}.\end{equation}
\end{proposition}

\begin{proof}Let $P(t,x,dz)$ be the transition probability function of $X_t$. We find
that for any $f\in B_b^+(\R^d)$ and $x$, $y\in\R^d$,
\begin{align*}P_tf(x)&=\int f(z)P(t,x,dz)=\int f(z)p_t(x,z)dz\\
&=\int f(z)\frac{p_t(x,z)}{p_t(y,z)}p_t(y,z)dz\\
&\le\int f(z) p_t(y,z)dz\,\max_{z\in\R^d}
\frac{p_t(x,z)}{p_t(y,z)}\\
&\le
CP_tf(y)\bigg(1+\frac{|x-y|}{t^{1/\alpha}}\bigg)^{d+\alpha}.\end{align*}
The proof is completed.
\end{proof}

\begin{remark}\label{stbelmark} According to the proof of Proposition \ref{levtsss}, the inequality \eqref{strongharnack}
is an consequence of \eqref{st2}. Thus, \eqref{strongharnack} is
also satisfied for symmetric (not necessarily rotationally
invariant) $\alpha$-stable L\'{e}vy processes, see e.g.\ (Bogdan et
al., 2003).
\end{remark}

Next, we turn to the proof of Theorem \ref{section1th1}.

\begin{proof}[Proof of Theorem \ref{section1th1}]

(1) We first assume that $A=0$.  Since $\nu(dz)\ge c|z|^{-d-\alpha}dz$, we
can write $$X_t=X_t'+X_t^{S},$$ where $X_t^{S}$ is a rotationally invariant symmetric $\alpha$-stable L\'{e}vy process with L\'{e}vy measure
$$\nu_{X^S}(dz):=c|z|^{-d-\alpha}dz,$$ and $X_t'$
is a L\'{e}vy process with L\'{e}vy measure
$$\nu_{X'}(dz):=\nu(dz)- c|z|^{-d-\alpha}dz\ge0.$$ Let $P_t$, $P_t^{S}$ and $P'_t$ be the semigroups of the L\'{e}vy processes $X_t$, $X_t^{S}$
and $X_t'$, respectively. Then, according to Proposition
\ref{levtsss}, there exists $C>0$ such that for any $t>0$, $x$, $y\in\R^d$ and $f\in
B_b^+(\R^d)$,
$$P^S_tf(x)\le
CP^S_tf(y)\bigg(1+\frac{|x-y|}{t^{1/\alpha}}\bigg)^{d+\alpha}.$$
Therefore, for $t>0$, $x$, $y\in\R^d$ and $f\in B_b^+(\R^d)$,
\begin{align*}P_tf(x)&=P^{S}_tP'_tf(x)\\
&\le
CP^{S}_tP'_tf(y)\bigg(1+\frac{|x-y|}{t^{1/\alpha}}\bigg)^{d+\alpha}\\
&=CP_tf(y)\bigg(1+\frac{|x-y|}{t^{1/\alpha}}\bigg)^{d+\alpha},\end{align*}
where in the above inequality we have used the fact that $P'_tf\in
B_b^+(\R^d)$. The desired assertion \eqref{section1th122222}
follows.

(2) Suppose that $A\neq0$. We borrow an idea from the proof of
(Wang, 2011, Theorem 1.1). It is well known that for the semigroup
of Ornstein-Uhlenbeck process
$$P_tf(x)=\int f(e^{tA}x+y)\mu_t(dy),$$ where $\mu_t$ is the probability measure on $\R^d$ with characteristic function
$$\hat{\mu}_t(\xi)=\exp\bigg[\int_0^t\Phi(e^{sA^*}\xi)ds\bigg],$$ and $\Phi$ is the symbol of the L\'{e}vy process
$L_t$, cf.\ see (Jacob, 2001). Define $$c_0=\int_{\{|z|\le
e^{-\|A\|}\}}\big(1-\cos z_1\big)|z|^{-d}dz.$$ Then, following Part
(III) in the proof of (Wang, 2011, Theorem 1.1), for any $s\in
(0,1]$,
$$\xi \mapsto\Phi(e^{sA^*}\xi)ds-c_0c|\xi|^\alpha$$ is a L\'{e}vy
symbol. Thus, for every $t\in(0,1]$ there exists a probability
measure $\pi_t$ on $\R^d$ with
$$\hat{\pi}_t(\xi)=\exp\bigg[\int_0^t\Phi(e^{sA^*}\xi)ds-tc_0c|\xi|^\alpha\bigg].$$
Let $P_t^S$ be the semigroup for the L\'{e}vy process with symbol $c_0c|\xi|^\alpha$, and define $$P'_t f(x)=\int f(x+z)\pi_t(dz).$$ We have
$$P_tf(x)=P_t^SP'_t f(e^{tA} x).$$ Note that $P_t^S$ is also the semigroup of a rotationally invariant symmetric $\alpha$-stable L\'{e}vy process
with L\'{e}vy measure
$$\nu_{S}(dz):=c_1|z|^{-d-\alpha}dz.$$ Then the required assertion for $t\in(0,1]$ follows form the arguments in step (1).

For any $t>0$ and $f\in B_b^+(\R^d)$,
\begin{align*}P_tf(x)&=P_{t\wedge 1}P_{(t-1)^+}f(x)\\
&\le C P_{t\wedge 1}P_{(t-1)^+}f(y)\bigg(1+\frac{|x-y|}{(t\wedge1)^{1/\alpha}}\bigg)^{d+\alpha}\\
&=P_tf(y)\bigg(1+\frac{|x-y|}{(t\wedge1)^{1/\alpha}}\bigg)^{d+\alpha},\end{align*}
where in the inequality above we have used again the fact that
$P_{(t-1)^+}f\in B_b^+(\R^d)$. The proof is end.
\end{proof}

To end up this section, we further present two remarks about Theorem \ref{section1th1}.
\begin{remark}\rm  (1) The proof of Theorem \ref{section1th1} yields that \eqref{section1th1222}
and \eqref{section1th12223333} hold for the following two classes of
L\'{e}vy type processes: (i) \emph{Stable-like processes on $\R^d$},
whose L\'{e}vy jump kernel is given by
$j(x,y)=\frac{c(x,y)}{|x-y|^{d+\alpha}}$ and $c(x,y)$ is a symmetric
function on $\R^d\times \R^d$ that is bounded upper and below by two
positive constants. More details about stable-like processes are
referred to (Chen and Kumagai, 2003, Theorem 1.1) and (Chen and
Kumagai, 2008, Theorem 1.2). (ii) \emph{Markov processes on $\R^d$
generated by $\triangle^{\alpha/2}+b(x)\cdot\bigtriangledown$},
where $1<\alpha<2$ and $b$ in the Kato class
$\mathscr{K}_d^{\alpha-1}$ on $\R^d$. The readers can see (Bogdan
and Jakubowski, 2007, Theorem 2) for the comparability between the
transition density for the small time of these processes and that of
rotationally invariant symmetric $\alpha$-stable L\'{e}vy processes.

(2) As shown in (Wang, 1997, 2006b, 2007, 2010a), to derive
contractivity properties of the semigroup from the Harnack
inequality, one needs concentration properties of the associated
invariant measure, which is however unknown for the present setting.
\end{remark}

\subsection{Proof of Theorem \ref{section1th2}}

We begin with recalling some results about a rotationally invariant symmetric truncated
$\alpha$-stable L\'{e}vy process $X_t$. The corresponding L\'{e}vy measure is
given by
\begin{equation*}\nu(dz)=c|z|^{-(d+\alpha)}\I_{\{|z|\le r\}}dz,\end{equation*}
where $c$, $r\in (0,\infty)$ and $\alpha\in(0,2).$ Let $p_t(x,y)$ be
the density function of $X_t$. Then, according to (Chen et al.,
2008, Proposition 2.1 and Theorems 2.3 and 3.6), there exist
positive constants $c_i>0$ $(i=1,\cdots,6)$ such that for any
$t\in(0,1]$ and $x$, $y\in\R^d$ with $|x-y|\le 1$,
\begin{equation}\label{tden1}c_2\bigg(t^{-d/\alpha}\wedge\frac{t}{|x-y|^{d+\alpha}}\bigg) \le p_t(x,y)\le c_1\bigg(t^{-d/\alpha}
\wedge\frac{t}{|x-y|^{d+\alpha}}\bigg),\end{equation} and for any
$t\in(0,1]$ and $x$, $y\in\R^d$ with $|x-y|> 1$,
\begin{equation}\label{tden2} c_5\bigg(\frac{t}{|x-y|}\bigg)^{c_6|x-y|}\le p_t(x,y)\le c_3\bigg(\frac{t}{|x-y|}\bigg)^{c_4|x-y|}.\end{equation}
Moreover, by (Chen et al., 2008, Proposition 2.2), there also exists
a constant $c_7>0$ such that for any $t\in(0,1]$ and $x$,
$y\in\R^d$,
\begin{equation}\label{tden3}p_t(x,y)\le c_7t^{-d/\alpha}.\end{equation}

%The argument of Theorem \ref{levt} also can yield Harnack inequality
%for OU processes driven by general L\'{e}vy process when sharp estimates for the
%corresponding transient density of L\'{e}vy process have been known. Refer to
%\cite[Theorem 1.2, Examples 2.3 and 2.4]{CK2} and \cite[Theorems 1.2
%and 1.4]{CK3} for existing results for Markov transition semigroups,
%whose L\'{e}vy jump kernel has isotropic upper and lower bounds in
%small scales.

\medskip

\begin{lemma}\label{levttt} Let $p_t(x,y)$ be the density function of the rotationally invariant symmetric truncated
$\alpha$-stable L\'{e}vy process $X_t$ above. Then, there are two
positive constants $C_1$ and $C_2$ such that for any $0<t\le 1$,
$x$, $y$ and $z\in\R^d$,
\begin{equation}\label{st3ttt4444}\frac{p_t(x,z)}{p_t(y,z)}\le C_1 t^{-d/\alpha}\Big(\frac{2\vee|x-y|\vee|y-z|}{t}\Big)
^{C_2(2\vee |x-y|\vee|y-z|)}.
\end{equation}
\end{lemma}

\begin{proof}  Fix $t\in (0,1]$. When $|y-z|\le t^{1/\alpha}$, by \eqref{tden1}, $$p_t(y,z)\ge c_2 t^{-d/\alpha}.$$ Thus, according to \eqref{tden3},
$$\frac{p_t(x,z)}{p_t(y,z)}\le \frac{c_7t^{-d/\alpha}}{c_2 t^{-d/\alpha}}=\frac{c_7}{c_2}.$$ Similarly, when $t^{1/\alpha}< |y-z|\le 1$, $$p_t(y,z)\ge
\frac{c_2t}{|x-y|^{d+\alpha}},$$which implies that
$$\frac{p_t(x,z)}{p_t(y,z)}\le \frac{c_7t^{-d/\alpha}}{c_2 t/|y-x|^{-(d+\alpha)}}=\frac{c_7}{c_2t^{d/\alpha+1}}.$$
Furthermore, by \eqref{tden2} and \eqref{tden3}, if $|y-z|\ge 1$,
then
$$\frac{p_t(x,z)}{p_t(y,z)}\le
\frac{c_7t^{-d/\alpha}}{c_5\big({t}/{|y-z|}\big)^{c_6|y-z|}}\le
\frac{c_7}{c_5t^{d/\alpha}}\Big(\frac{|y-z|}{t}\Big)^{c_6|y-z|}.$$
In particular, when $1\le |y-z|<2\vee|x-y|$,
$$\frac{p_t(x,z)}{p_t(y,z)}\le\frac{c_7}{c_5t^{d/\alpha}}\Big(\frac{|y-z|}{t}\Big)^{c_6|y-z|}\le
\frac{c_7}{c_5t^{d/\alpha}}\Big(\frac{2\vee|x-y|}{t}\Big)^{c_6(2\vee|x-y|)}.$$
Combining with all the estimates above, for any $0<t\le 1$, $x$, $y$
and $z\in\R^d$,
\begin{equation*}\label{st3ttt}\frac{p_t(x,z)}{p_t(y,z)}\le \begin{cases}C_1 t^{-d/\alpha}\Big(\frac{2\vee|x-y|}{t}\Big)^{C_2(2\vee |x-y|)} ,& |y-z|\le 2\vee|x-y|,\\
\qquad C_3 t^{-d/\alpha}\Big(\frac{|y-z|}{t}\Big)^{C_4|y-z|},
 &|y-z|> 2\vee|x-y|.\end{cases}
\end{equation*}
It follows our desired assertion.
\end{proof}

\begin{proposition}\label{levtsss4545} Let $P_t$ be the semigroup of the rotationally invariant symmetric truncated
$\alpha$-stable L\'{e}vy process $X_t$ above. Then for any $t>0$,
$x$, $y\in\R^d$ and $f\in B_b^+(\R^d)$ with $f\ge1$,
\begin{equation}\label{logstrongharnack}P_t(\log f)(x)\le \log P_tf(y)+C\big(1+|x-y|\big)\log\bigg(\frac{2+|x-y|}{t\wedge1}\bigg).\end{equation}
\end{proposition}

\begin{proof} (1) For any $f\in B_b^+(\R^d)$ with $f\ge 1$ and $t>0$,
\begin{equation}\label{proftrstable1} \aligned P_t(\log f)(x)&=\int\! \big(\log f(z)\big) p_t(x,z)\, dz\\
&=\int\!\big(\log f(z)\big)\, \frac{p_t(x,z)}{p_t(y,z)}\,p_t(y,z)dz\\
&\le\displaystyle\int\!\log\bigg(\frac{p_t(x,z)}{p_t(y,z)}
\bigg)\,p_t(x,z)dz+\log\! \bigg(\int
f(z)p_t(y,z)dz\bigg)\\
&= \displaystyle\int\log\bigg(\frac{p_t(x,z)}{p_t(y,z)}\bigg)
p_t(x,z)dz+\log P_tf(y) ,\endaligned\end{equation} where in the
above inequality we have used the following Young inequality: for
any probability measure $\mu$, if $g$, $h\ge 0$ with $\mu(g)=1$,
then
$$\mu(gh)\le \mu (g\log g)+\log\mu(e^{h}).$$

(2) According to \eqref{st3ttt4444}, for any $0<t\le 1$ and $x$,
$y,$ $z\in\R^d$,
$$\log \frac{p_t(x,z)}{p_t(y,z)}\le \log C_1+\frac{d}{\alpha}\log\frac{1}{t}+C_2\big(2\vee|x-y|\vee|y-z|\big)\log \bigg(\frac{2\vee|x-y|\vee|y-z|}{t}\bigg). $$ On the
other hand, if $|y-z|\ge 2\vee|x-y|$, then $$|x-z|\ge |y-z|-|x-y|\ge
|y-z|/2\ge 1,$$ and so, by \eqref{tden2},
$$p_t(x,z)\le c_3\bigg(\frac{t}{|x-z|}\bigg)^{c_4|x-z|}\le
c_3\bigg(\frac{2t}{|y-z|}\bigg)^{c_4|y-z|/2}.$$
Therefore,\begin{equation}\label{proftrstable2}\aligned
\displaystyle\int\!&\log\!
\bigg(\frac{p_t(x,z)}{p_t(y,z)}\bigg)p_t(x,z)dz\\
\le &\log
C_1+\frac{d}{\alpha}\log \frac{1}{t}\\
&+C_2\big(2\vee|x-y|\big)\log
\bigg(\frac{2\vee|x-y|}{t}\bigg)\int_{|y-z|\le
2\vee|x-y|}p_t(x,z)dz\\
&+C_2\int_{|y-z|> 2\vee|x-y|}|y-z|\log
\bigg(\frac{|y-z|}{t}\bigg)p_t(x,z)dz\\
\le &\log C_1+\frac{d}{\alpha}\log
\frac{1}{t}+C_2\big(2\vee|x-y|\big)\log
\bigg(\frac{2\vee|x-y|}{t}\bigg)\\
&+C_3\int_{|y-z|>
2\vee|x-y|}|y-z|\bigg(\frac{2t}{|y-z|}\bigg)^{c_4|y-z|/2}\log
\bigg(\frac{|y-z|}{t}\bigg)dz\\
\le &\log C_1+\frac{d}{\alpha}\log
\frac{1}{t}+C_2\big(2\vee|x-y|\big)\log
\bigg(\frac{2\vee|x-y|}{t}\bigg)\\
&+C_4t^{c_4}\big(1+\log t^{-1}\big)\int_{|y-z|>
2}|y-z|^{-c_4|y-z|/2+1}\log |y-z| dz\\
\le &C\big(1+|x-y|\big)\log\bigg(\frac{2+|x-y|}{t}\bigg).
\endaligned\end{equation}

(3) For $t\in (0,1]$, the required assertion follows from
\eqref{proftrstable1} and \eqref{proftrstable2}. For any $t>0$ and
$f\in B_b^+(\R^d)$ with $f\ge 1$,
\begin{align*}P_t(\log f)(x)&=P_{t\wedge 1}P_{(t-1)^+}(\log f)(x)\\
&\le P_{t\wedge 1}(\log P_{(t-1)^+}f)(x)\\
 &\le\log \big(P_{t\wedge
1}P_{(t-1)^+}f\big)(y)+C\big(1+|x-y|\big)\log\bigg(\frac{2+|x-y|}{t\wedge1}\bigg)\\
&=\log
P_tf(y)+C\big(1+|x-y|\big)\log\bigg(\frac{2+|x-y|}{t\wedge1}\bigg),\end{align*}
where the first inequality follows from the Jensen inequality: for
any probability measure $\mu$ and $f\ge0$, $\mu(\log f)\le \log
\mu(f),$ and in the second inequality we have used the fact that
$P_{(t-1)^+}f\ge1$ and the conclusion for $t\in(0,1]$ proved in the
previous two steps. We have proved our desired assertion.
\end{proof}

\begin{remark}By the proof of Proposition \ref{levtsss4545},
the inequality \eqref{logstrongharnack} is based on estimates
\eqref{tden1}, \eqref{tden2} and \eqref{tden3}. Therefore,
\eqref{logstrongharnack} also holds for L\'{e}vy type processes,
whose L\'{e}vy jump kernel is given by
$j(x,y)=\frac{c(x,y)}{|x-y|^{d+\alpha}}\I_{\{|x-y|\le r\}}$. Here,
$r\in(0,\infty)$ and $c(x,y)$ is a symmetric function on $\R^d\times
\R^d$ that is bounded upper and below by two positive constants, see
e.g.\ (Chen et al., 2008).\end{remark}

\medskip

We finally turn to the proof of Theorem \ref{section1th2}.

\begin{proof}[Proof of Theorem \ref{section1th2}]Under the condition in Theorem \ref{section1th2}, we
can write $$X_t=X_t'+X_t^{T},$$ where $X_t^{T}$ is a rotationally invariant symmetric truncated $\alpha$-stable L\'{e}vy process with L\'{e}vy measure
$$\nu_{X^T}(dz):=c|z|^{-d-\alpha}\I_{\{|z|\le r\}}dz,$$ and $X_t'$
is a L\'{e}vy process with L\'{e}vy measure
$$\nu_{X'}(dz):=\nu(dz)- c|z|^{-d-\alpha}\I_{\{|z|\le r\}}dz\ge0.$$ Let $P_t$, $P_t^{T}$ and $P'_t$ be the semigroups of the L\'{e}vy processes $X_t$, $X_t^{T}$
and $X_t'$, respectively. Then, according to Proposition
\ref{levtsss4545}, there exists $C>0$ such that for any $t>0$, $x$,
$y\in\R^d$ and $f\in B_b^+(\R^d)$ with $f\ge1$,
$$P^T_t(\log f)(x)\le \log P^T_tf(y)+C\big(1+|x-y|\big)\log\bigg(\frac{2+|x-y|}{t\wedge1}\bigg).$$
Therefore, for $t>0$ and $f\in B_b^+(\R^d)$ with $f\ge1$,
\begin{align*}P_t(\log f)(x)&=P^{T}_tP'_t(\log f)(x)\\
&\le
P^{T}_t\log (P'_tf)(x)\\
&\le C\log (P^{T}_tP'_t)f(y)+C\big(1+|x-y|\big)\log\bigg(\frac{2+|x-y|}{t\wedge1}\bigg)\\
&=C\log
P_tf(y)+C\big(1+|x-y|\big)\log\bigg(\frac{2+|x-y|}{t\wedge1}\bigg)
,\end{align*} where in the above two inequalities we have used the
Jensen inequality and the fact that $P'_tf\ge1$, respectively. The
proof is complete.
\end{proof}

\bigskip

\noindent{\bf Acknowledgement.} The author would like to thank
Professor Feng-Yu Wang for comments on earlier versions of the paper,
and to thank Professor Ren\'{e} L. Schilling for helpful
conversations. Financial support through the Alexander-von-Humboldt
Foundation and the Natural Science Foundation of Fujian $($No.\!
2010J05002$)$ is also gratefully acknowledged. He also would like
to thank an anonymous referee for carefully corrections on the first
draft of this paper.

\end{document}